\documentclass[11pt]{amsart}
\usepackage{bm}
\usepackage{fullpage}
\usepackage{amssymb}
\usepackage{amsmath,amsfonts,amsthm}
\usepackage{hyperref}
\usepackage{color}
\usepackage{soul}
\usepackage{cite}
\usepackage{tikz}
\definecolor{uuuuuu}{rgb}{0.26666666666666666,0.26666666666666666,0.26666666666666666}

\newtheorem{theorem}{Theorem}[section]
\newtheorem{problem}{Problem}[theorem]

\newtheorem{lemma}[theorem]{Lemma}

\newtheorem{claim}[theorem]{Claim}
\newtheorem{proposition}[theorem]{Proposition}
\newtheorem{corollary}[theorem]{Corollary}

\newtheorem{alphtheorem}{Theorem}

\newtheorem{alphlemma}{Lemma}

\DeclareMathOperator\G{\mathcal{G}}

\DeclareMathOperator\F{\mathcal{F}}

\DeclareMathOperator\C{\mathcal{C}}

\def\isdef{\mbox {$\ \stackrel{\rm def}{=} \ $}}
\title{ Vertex Partitions and	Maximum $\G$-free Subgraphs}

\author{Yaser Rowshan$^1$}
\address{$^1$Y. Rowshan, 
	Department of Mathematics, Institute for Advanced Studies in Basic Sciences (IASBS), Zanjan 45137-66731, Iran.}
\email{y.rowshan@iasbs.ac.ir,~~~y.rowshan.math@gmail.com.}

\begin{document}
	\maketitle 	
	\begin{abstract}
		We define a $(V_1, V_2, \ldots, V_k)$-partition for a given graph $H$ and graphical properties $P_1, P_2, \ldots, P_k$ as a partition where each $V_i$ induces a subgraph of $H$ with property $P_i$. In 1979, Bollob'{a}s and Manvel demonstrated that if a graph $H$ has a maximum degree $\Delta(H)\geq 3$ and clique number $\omega(H)\leq \Delta(H)$, with $\Delta(H)=p+q$, there exists a $(V_1, V_2)$-partition of $V(H)$. This partition ensures that $\Delta(H[V_1])\leq p$, $\Delta(H[V_2])\leq q$, $H[V_1]$ is $(p-1)$-degenerate, and $H[V_2]$ is $(q-1)$-degenerate. Matamala (2007) extended this result by showing that for any graph $H$ with $\Delta(H)=p+q$, there exists a $(V_1, V_2)$-partition of $V(H)$ where $H[V_1]$ is a maximum order $(p-1)$-degenerate induced subgraph and $H[V_2]$ is $(q-1)$-degenerate. Additionally, Catlin and Lai proved that if $\Delta(H)\geq 5$, $H$ has a $(V_1, V_2)$-partition such that $H[V_1]$ is a maximum order acyclic induced subgraph, $\omega(H[V_2])\leq \Delta(H)-2$, and $\Delta(H[V_2])\leq \Delta(H)-2$.
		
		Rowshan and Taherkhani demonstrated that given a graph $G$ with a minimum degree $\delta(G)$ and for $k=\lceil \frac{\Delta(H)}{\delta(G)}\rceil$, there exists a $(V_1, V_2, \ldots, V_k)$-partition of the vertex set of $H$, such that each $H[V_i]$ is $G$-free, meaning it does not contain a subgraph isomorphic to $G$, and $H[V_1]$ is a maximum order $G$-free induced subgraph.
		
		In our paper, we present a novel result for a connected graph $H$ with $\Delta(H)\geq 5$ and without $K_{\Delta(H)+1}\setminus e$ as a subgraph. We establish that when $p_1\geq p_2\geq\cdots\geq p_{k-1}\geq 2$, $p_k\geq 4$, $\sum_{i=1}^k p_i=\Delta(H)-1+k$, and $\mathcal{G}_i$ represents a family of graphs with a minimum degree at least $p_i-1$ for each $i\in [k-1]$, a $(V_1, V_2, \ldots, V_k)$-partition of $V(H)$ exists. This partition guarantees that $H[V_1]$ is a maximum order $\mathcal{G}_1$-free induced subgraph, $H[V_i]$ is $\mathcal{G}_i$-free for each $2\leq i\leq k-1$, $\Delta(H[V_k])\leq p_k$,  and either $H[V_k]$ is $K_{p_k}$-free or its $p_k$-cliques are disjoint.
	\end{abstract}
	
	\section{Introduction} 
	In this article, all graphs under consideration are finite, undirected, and simple. For a given graph $H=(V(H),E(H))$, the degree of a vertex $v\in V(H)$ is denoted as $\deg_H(v)$ (or simply $\deg(v)$), and its set of neighbors is represented by $N_H(v)$ (or $N(v)$). The maximum degree of graph $H$ is denoted as $\Delta(H)$, and the minimum degree as $\delta(H)$. When referring to a subset $W$ of $V(H)$, the induced subgraph on $W$ is denoted as $H[W]$. For two disjoint subsets $V_1$ and $V_2$ of $V(H)$, the set $E(V_1, V_2)$ represents all the edges $vv' \in E(H)$, where $v\in V_1$ and $v'\in V_2$. The clique number $\omega(H)$ of a graph $H$ is defined as the largest integer $k$ for which $H$ contains a complete subgraph of size $k$. The join of two graphs $G$ and $H$ is denoted as $G\oplus H$ and is obtained by connecting each vertex of $G$ to every vertex in $H$. Furthermore, when referring to an edge $e$ in graph $G$, we use $G\setminus e$ to denote the graph resulting from the removal of $e$ in $G$.
	
	We define a $(V_1,V_2,\ldots,V_k)$-partition for a given graph $H$ and graphical properties $P_1,P_2,\ldots,P_k$ such that each subgraph induced on $V_i$ satisfies property $P_i$. Further references on $(V_1,\ldots,V_k)$-partition can be found in \cite{simoes1976joins, bickle2021maximal, bickle2021wiener, bickle2012structural, von2022point, lick1970k}.
	
	A graph $H$ is considered $k$-degenerate if every subgraph of $H$ contains a vertex with degree at most $k$. Specifically, when the property $P_i$ implies that $H[V_i]$ is $p_i$-degenerate for some positive integer $p_i$, refer to \cite{bickle2021maximal, bickle2021wiener}. In the case where $k=2$, Bollob'{a}s and Manvel \cite{Bollob} have presented the following result concerning $(V_1,V_2)$-partition.
	\begin{alphlemma}\label{lem1}{\rm\cite{Bollob}}
		Assume that $H$ is a graph with a maximum degree $\Delta(H)\geq 3$ and a clique number $\omega(H)\leq \Delta(H)$. If $\Delta(H)=p+q$, it can be shown that there exists a $(V_1,V_2)$-partition of the vertex set $V(H)$. This partition satisfies the following properties: $\Delta(H[V_1])\leq p$, $\Delta(H[V_2])\leq q$, $H[V_1]$ is $(p-1)$-degenerate, and $H[V_2]$ is $(q-1)$-degenerate. 
	\end{alphlemma}	
	As an extension of Lemma \ref{lem1}, Catlin and Lai \cite{Catlin1} later proved the following theorem..
	\begin{alphtheorem}\label{th1}{\rm\cite{Catlin1}}
		Assuming that $H$ is a graph with $\Delta(H)=d\geq 3$ and a clique number $\omega(H)\leq \Delta(H)$, it can be shown that $H$ possesses a $(V_1, V_2)$-partition satisfying the following properties:
		\begin{itemize}
			\item  For $d = 3$, $V_1$ is a maximum independent set and $H[V_2]$ is acyclic.
			\item  For $d = 4$, $H[V_1]$  is a maximum acyclic induced subgraph and $H[V_2]$ is acyclic.
			\item  For $d\geq 5$, $H[V_1]$ is a  maximum acyclic induced subgraph,  $\omega(H[V_2])\leq d-2$, and $\Delta(H[V_2])\leq d-2$.
			
		\end{itemize}
	\end{alphtheorem}
	The result presented below is closely related to Lemma \ref{lem1} and Theorem \ref{th1}. It was established by Matamala in \cite{MR2327961}.
	\begin{alphtheorem}\label{thm:mat}{\rm\cite{MR2327961}}
		Assume that $H$ is a graph with  $\Delta(H)\geq 3$ and clique number $\omega(H)\leq \Delta(H)$. If $\Delta(H)= p+q$  then there exists a $(V_1,V_2)$-partition of $V(H)$, such that $H[V_1]$ is a maximum order $(p-1)$-degenerate induced subgraph of  $H$ and $H[V_2]$ is $(q-1)$-degenerate subgraph.
	\end{alphtheorem}
	In an extension of Brooks' Theorem, Catlin demonstrated that any graph $H$ with $\Delta(H)\geq 3$ and without $K_{\Delta(H)+1}$ as a subgraph can be colored with $\Delta(H)$ colors such that one of the color classes forms a maximum independent set \cite{Catlin}. Let $\G$ be a family of graphs. We define $H$ as $\G$-free if it does not contain any subgraph isomorphic to $G$ for every $G\in \G$. As an analogy to Catlin's result, the author and Taherkhani \cite{rowshan2022catlin} established the following result:
	\begin{alphtheorem}\label{2th}\cite{rowshan2022catlin}
		Let $d_1,\ldots,d_k$ be $k$  positive integers. Assume that $G_1,\ldots,G_k$
		are  connected graphs with minimum degrees  $d_1,\ldots,d_k$, respectively, and $H$ is a connected graph with maximum degree
		$\Delta(H)$ where $\Delta(H)=\sum_{i=1}^{k}d_k$. Assume that $G_1,G_2,\ldots,G_k$,
		and $H$ satisfy the following conditions:
		\begin{itemize}
			\item If $k=1$, then $H$ is not isomorphic to  $G_1$.
			\item If  $G_i$ is isomorphic to $K_{d_i+1}$ for each $1\leq i\leq k$, then $H$ is not isomorphic to $K_{\Delta(H)+1}$.
			\item If  $G_i$ is isomorphic to $K_{2}$ for each $1\leq i\leq k$, then $H$ is neither an odd cycle nor a complete graph.
		\end{itemize}
		Then, there is a partition of vertices of $H$ to $V_1,\ldots,V_k$
		such that each $H[V_i]$ is $G_i$-free and moreover one of  $V_i$s  can be chosen in a way that $V_i$ has maximum possible size  such that for which we have $H[V_i]$  be a  $G_i$-free subgraph in $H$. 
	\end{alphtheorem}
	Also the author and Taherkhani \cite{yas} established the following result:
	\begin{alphtheorem} \label{3th}\cite{yas}
		Suppose that $H$ is a graph  $\omega(H)\leq \Delta(H)-1$. Let $k\geq 2$ be a positive integer.
		Assume  that  $p_1\geq p_2\geq\cdots\geq p_k\geq 2$  are  $k$ positive integers and  $\sum_{i=1}^k p_i=\Delta(H)-1+k$. If $p_1+p_2\geq 7$, then
		there exists a partition of $V(H)$ into $V_1,V_2,\ldots, V_k$
		such that for each $1\leq i\leq k$, $H[V_i]$ is $K_{p_i}$-free.
	\end{alphtheorem}
		\begin{alphtheorem} \label{4th}\cite{yas}
		Assume that $H$  is a  graph with  $\Delta(H)\geq 6$ and  clique number $\omega(H)$ where $4\leq \omega(H)\leq \Delta(H)-2$.
		Denote  $\omega(H)=p$ and  $\Delta(H)+1-p=q$.
		Then there exists  $V_1\subseteq V(H)$ such that $V_1$ is a maximum $K_p$-free subset of $H$,
		and $H[V\setminus V_1] $ is $K_q$-free.	
	\end{alphtheorem}	
	Consider a connected graph $H=(V, E)$ with a maximum degree $d \geq 3$, which is distinct from $K_{d+1}$. Let $k \geq 2$ be a positive integer, and let $p_1,\ldots, p_k \geq 0$ be $k$ integers. We define $H$ as $(p_1, \ldots , p_k)$-partitionable if there exists a partition of $V(H)$ into sets $V_1,\ldots, V_k$ such that $H[V_i]$ is $p_i$-degenerate for $i \in [k]$. Abu-Khzam, Feghali, and Heggernes have established the following two results concerning $(p_1, \ldots , p_k)$-partitionable graphs.
	\begin{alphtheorem}{\rm\cite{abu2020partitioning}}
		Suppose that $H = (V, E)$ is a connected graph with  $\Delta(H)=d \geq 3$ distinct from $K_{d+1}$. For all integers $k \geq 2$ and $1\geq p_1,\ldots, p_k \geq 0$, such that $\sum_{i=1}^{k}p_i\geq d-k$, a $(p_1,\ldots, p_k)$-partition  of $H$ can be found in $O(|V|+ |E|)$-time.
	\end{alphtheorem}
	\begin{alphtheorem}{\rm\cite{abu2020partitioning}}
		For every integer $d \geq 5$ and every pair of non-negative integers $(p, q)$, so that $(p, q)\neq (1, 1)$
		and $p+q= d-3$, deciding whether a graph with maximum degree $d$ is $(p, q)$-partitionable is NP-complete.
	\end{alphtheorem}
	As a related result of Lemma \ref{lem1}, Theorem \ref{th1}, Theorem \ref{thm:mat},  and Theorem \ref{2th}, in this article, we prove the  following theorem.
	
	\begin{theorem}[Main result]\label{mth1} Suppose that $H = (V, E)$ is a connected graph with maximum degree $\Delta(H)\geq 5$  and  $H$ is $K_{\Delta(H)+1}\setminus e$-free. Suppose that $p$ and $q$  are two positive integers, such that $p\geq 2, q\geq 4$ and  $\Delta(H)+1=p+q$. Set $\G$ as a collection of   graphs with minimum degree at least $p-1$. Then there exists a $(V_1,V_2)$-partition of $V(H)$, such that $H[V_1]$  is a maximum order $\G$-free induced subgraph of $H$,   $\Delta(H[V_2])\leq q$, and  either $H[V_2] $ is $K_q$-free subgraph or its $q$-cliques  are disjoint. 	
	\end{theorem}
	By utilizing induction on $k$, we can demonstrate that the following result holds as a generalization of Theorem \ref{mth1}. 	
	\begin{corollary}\label{cor1}
		Assume that $H = (V, E)$ is a connected graph with a maximum degree $\Delta(H)\geq 5$. Let $p_1\geq p_2\geq\cdots \geq p_{k-1}\geq 2$ and $p_k\geq 4$ be $k$ positive integers such that $\sum_{i=1}^k p_i=\Delta(H)-1+k$. Furthermore, let $\G_i$ denote a collection of graphs with a minimum degree of at least $p_i-1$ for each $i\in [k-1]$. If $H$ is free of $K_{\Delta(H)+1}\setminus e$, then there exists a $(V_1,V_2,\ldots, V_k)$-partition of $V(H)$ satisfying the following properties: $H[V_1]$ is a maximum order $\G_1$-free graph, $H[V_i]$ is $\G_i$-free for each $2\leq i\leq k-1$, $\Delta(H[V_k])\leq p_k$, and either $H[V_k]$ is $K_{p_k}$-free or its $p_k$-cliques are disjoint.
	\end{corollary}	
	Theorem \ref{2th} implies that for every graph $H$ with a maximum degree $\Delta(H)\geq 5$ and without $K_{\Delta(H)+1}\setminus e$ as a subgraph, if $\Delta+1=2p$ for some positive integer $p$, then there exists a $K_p$-free $\lceil \frac{\Delta}{p-1}\rceil$-coloring of $H$ such that one of its color classes is a maximum induced $K_p$-free subgraph in $H$. For instance, for $\Delta(H)=9$ and $p=5$, Theorem \ref{2th} guarantees a $K_5$-free $3$-coloring of $H$ such that one of its color classes is a maximum order induced $K_5$-free subgraph in $H$. However, Theorem \ref{mth1} states that there is a $2$-coloring of $H$ such that one of its color classes is a maximum order induced $K_5$-free subgraph in $H$ and $H[V_2]$ is either $K_5$-free or the $5$-cliques of $H[V_2]$ are disjoint. Clearly, the result of Theorem \ref{mth1} is stronger than that of Theorem \ref{2th}. Moreover, if we consider $\G =\C={C_n,~n\geq 3}$, $p=3$, and $q=\Delta(H)-2$, then Theorem \ref{mth1} coincides with Theorem \ref{th1}. It is worth noting that by incorporating some results from Demetres Christofides, Katherine Edwards, and Andrew D. King in \cite{christofides2013note}, if we replace the assumption $q\geq 4$ with $q\geq 2$ in Theorem \ref{mth1}, we obtain the same result, but we lose the maximality of $|V_1|$ in Theorem \ref{mth1}. In other words, we have the following result.

	\begin{lemma}\label{ml1}Consider a connected graph $H$ with $\Delta(H)\geq 7$ that is $K_{\Delta(H)}\setminus e$-free, and let $p$ and $q$ be two integers satisfying $p,q\geq 2$ and $\Delta(H)+1=p+q$. Then there exists a $(V_1,V_2)$-partition of $V(H)$ such that $H[V_1]$ is $K_p$-free, and either $H[V_2]$ is $K_q$-free or its $q$-cliques are disjoint.
	\end{lemma}
	\section{Vertex Partitions and	Maximum $R$-free Subgraphs}	
	In this section, we prove  Theorem \ref{mth1}  for the case that $\G$ has only one member, say  $R$, where $R$ is a   graph with minimum degree at least $p-1~(p\geq 2)$.
	\begin{theorem}\label{M.th} Assume that $H$  is a connected graph with
		$\Delta(H)=d\geq 5$ and $p$ and $q$ be two positive integers, where $p\geq2,  q\geq 4$ and $d+1=p+q$. Also, let $G\cong K_{q+1}\setminus e$. Suppose that $\F$ consists of $S\subseteq V(H)$ for which:
		\begin{itemize}
			\item  $S$ has the maximum possible size (M-P-size) such that $H[S]$ is $R$-free, in  other words  $S$ is a maximum order $R$-free induced subgraph of $H$.
		\end{itemize}
		Now, for each $S\in \F$ the following statements hold:
		\begin{itemize}
			\item[(I)]: For each vertex  $v\in  \overline{S}$, we have $|N(v)\cap S|\geq p-1$, $\Delta(H[\overline{S} ])\leq q$, and  the induced subgraph $H[S\cup\{v\}]$ has at least one copy  of  $R$. 
			\item[(II)]: Every vertex  $v\in\overline{S} $  lies in either  at most one copy of $G$ in $H\setminus S$ or  a copy of $K_{q+1}$ which is a connected component of $H[\overline{S}]$.
			\item[(III)]:  For any member  $S$ of $\F$ if $\overline{S}$ has a copy of $G$, then  $K_{d+1}\setminus e\subseteq H$.
		\end{itemize}
	\end{theorem}
	\begin{proof} Suppose that  $S$ be an arbitrary member of $\F$.
		
		{\bf Proof (I)}:Since the size of the set $S$ is maximal, for every vertex $v \in \overline{S}$, the graph $H[S \cup {v}]$ contains at least one copy $R$ that includes the vertex $v$. Therefore, we can conclude that $|N(v) \cap S| \geq p-1$ and $|N(v) \cap \overline{S}| \leq q$ for every vertex $v$.
		
		{\bf Proof (II)}:Assume by contradiction that there exists a vertex $v$ in $\overline{S}$ that is present in at least two copies of $G$ but does not belong to any copy of $K_{q+1}$ in $H[\overline{S}]$. Let $G_1$ and $G_2$ be two copies of $G$ in $H[\overline{S}]$ that contain $v$. Let $X=\{x_1, x_2, \ldots, x_{q-1}\}$ and $X'=\{x'_1, x'_2, \ldots, x'_{q-1}\}$ be the sets of $(q-1)$ vertices from $V(G_1)$ and $V(G_2)$, respectively, such that for each $1\leq i\leq q-1$, $\deg_{H[\overline{S}]}(x_i)=\deg_{H[\overline{S}]}(x'_i)=q$.
		
		If $|X\cap X'|\leq q-3$, then we can conclude that $|N(v)\cap \overline{S}|\geq q+1$, which implies $|N(v)\cap S|\leq p-2$. However, this contradicts part (I) of the result. Therefore, we assume that $|X\cap X'|\geq q-2 \geq 2$. Let $v'$ be a vertex in $X\cap X'$.
		
		Since $G_1\neq G_2$, we can deduce that $|N(v')\cap \overline{S}|\geq q+1$, which implies $|N(v')\cap S|\leq p-2$. However, this contradicts part (I) of the result. Thus, our initial assumption that there exists such a vertex $v$ is false, and the claim holds.\\
		{\bf Proof (III)}:  Suppose that $p\geq 2$ and $q\geq 4$. We take a $S\in \F$
		for which 
		\begin{itemize}
			\item(P1): $H[\overline{S}]$ has the least possible number of copies $G$,
			and, subject to that, 
			\item(P2): $E(H[\overline{S}])$ has the least possible size.
			
		\end{itemize}
		Let's assume that $G'$ is a copy of $G$ in $H[\overline{S}]$. We define $V' = V(G') = \{v'_0, v'_1, \ldots, v'_{q}\}$ such that $G'[V'\setminus \{v'_{q-1}, v'_{q}\}] \cong K_{q-1}$ and $G' \cong K_{q-1}\oplus H[{v'_{q-1}, v'_{q}}]$. It should be noted that $v'_{q-1}$ and $v'_{q}$ can be adjacent in $H$, in which case $H[V(G')]$ is isomorphic to $K_{q+1}$.
		Now, let's define $B'$ as follows:
		\[ B'\isdef\{v'_0,v'_1,\ldots, v_{q-2}'\}=V'\setminus\{v'_{q-1}, v'_{q}\}.\]
		Since $q \geq 4$, we have $|B'| \geq 2$. 	By the maximality of $S$ and using properties (I) and (II), we can conclude that for each $v \in B'$, $H[S \cup {v}]$ contains at least one copy of $R$. Since $v \in B'$, we can easily observe that $|N(v) \cap \overline{S}| \geq q$. Consequently, $|N(v) \cap S| \leq p-1$.
		The fact that $v$ lies in at least one copy of $R$ in $H[S \cup \{v\}]$ implies that $|N(v) \cap S| = p-1$. Therefore, we have $|N(v) \cap \overline{S}| = q$. If this were not the case, then $\deg(v) \geq p+q = d+1$, which would be a contradiction to the assumption that $\Delta(H) = d$.
		
		\begin{claim}\label{c1}
			Suppose that $C$ is a connected component of $H[S\cup \{v\}]$ containing  $v$. Then $C$  is  $(p-1)$-regular graph and  isomorphic to $R$ . 
		\end{claim}
		\begin{proof}[\bf Proof of Claim~\ref{c1}]  $C$ contains at least one copy of $R$, denoted as $R'$. We want to show that $C = R'$.
			
			Suppose, for the sake of contradiction, that $C \neq R'$. Without loss of generality, assume that $|V(C)| \geq |V(R')|+1$, which implies that there exists at least one vertex $v'$ in $C$ such that $v' \notin R'$.
			
			As $C$ is a connected component of $H[S \cup \{v\}]$, let us consider the distance $d_C(v,v')$ between $v$ and $v'$ in $C$. Note that since $|N(v) \cap S| = p-1$, all copies of $R$ in $C$ must contain the neighbourhood $N(v) \cap S$.
			
			Now, let $i \geq 1$ be the largest integer such that for every vertex $u \in C$ with $d_C(v,u) = i$, $u$ is included in all copies of $R$ in $C$. Since $v' \notin R'$, there exists at least one vertex $w$ in $R' \cap C$ such that $d_C(v,w) = i+1 \leq d_C(v,v')$, and $w \notin V(R')$. Consequently, there exists at least one neighbour $y$ of $w$ in $C$ such that $d(v,y) = d(v,w) - 1 = i$.
			
			As $d(v,y) = i$, vertex $y$ is included in all copies of $R$ in $C$. Let $S_1 = (S \cup \{v\}) \setminus \{y\}$. Note that $|S_1| = |S|$ and $H[S_1]$ is $R$-free because $y$ lies in all copies of $R$ in $H[S \cup \{v\}]$. Since $y$ is included in at least one copy of $R$ in $H[S \cup \{v\}]$, and at least one of them does not include $w$, we have $|N(y) \cap S_1| = |N(y) \cap (S \cup \{v\})| \geq p$. Therefore, $|N(y) \cap \overline{S}_1| \leq q-1$, which implies that $y$ lies in at most one copy of $G$ in $H \setminus S_1$. If $y$ is not contained in any copy of $G$ in $H \setminus S_1$, it contradicts property (P1).
			
			Therefore, we can assume that $y$ lies in one copy of $G$, and the number of copies of $G$ in $H \setminus S_1$ is equal to the number of copies of $G$ in $H \setminus S$. Since $|N(v) \cap \overline{S}| = q$ and $|N(y) \cap \overline{S}_1| = q-1$, it can be verified that $|E(H[\overline{S}_1])| \leq |E(H[\overline{S}])| - 1$. This contradicts property (P2). Hence, we conclude that $V(C) = V(R)$.
			
			Assume that all copies of $R$ in $C$ have the same vertex set $V(C)$. If there are at least two distinct copies of $R$ in $C$ with a common vertex set, then there exists a vertex $u \in R \subseteq C$ such that $|N(u) \cap V(C)| \geq p$. Define $S_1 = (S \cup \{v\}) \setminus \{u\}$, and the proof follows similarly as in the previous paragraph. Therefore, $C \cong R$.
			
			Now we shall show that $C$ is $(p-1)$-regular. Assume that there exists a vertex $y$ in $C$ with more than $p-1$ neighbour in $C$. Then, $|N(y) \cap S_1| = |N(y) \cap (S \cup \{v\})| \geq p$. This implies that $|N(y) \cap \overline{S}_1| \leq q-1$, which means that $y$ lies in at most one copy of $G$ in $H \setminus S_1$. If $y$ is not contained in any copy of $G$ in $H \setminus S_1$, it contradicts property (P1). Therefore, we can assume that $y$ lies in one copy of $G$, and the number of copies of $G$ in $H \setminus S_1$ is equal to the number of copies of $G$ in $H \setminus S$. However, it can be shown that $|E(H[\overline{S}_1])| \leq |E(H[\overline{S}])| - 1$, which contradicts property (P2).
			
			Hence, we conclude that $C$ is a $(p-1)$-regular graph.	
		\end{proof}
		
		Assuming that $R'$ is a copy of $R$ in $H[S\cup \{v\}]$, we can consider a vertex $y$ belonging to $V(R')\setminus \{v\}$, where $V(R')\setminus \{v\}\subseteq S$. Based on this, we can establish the following claim.

		\begin{claim}\label{c2}
			There is a copy of $G$ in  $H[(\overline{S}\setminus \{v\})\cup\{y\} ]$, which contains  $y$ and $|N(y)\cap (\overline{S}\setminus\{v\})|= q$.
		\end{claim}
		
		\begin{proof}[\bf proof of Claim \ref{c2}]
			According to Claim \ref{c1}, $R'$ is a $(p-1)$-regular graph and one of the connected components in $H[S\cup \{v\}]$. Consequently, for any vertex $y$ in $V(R')$, we have $|N(y)\cap (S\cup \{v\})|=p-1$, which implies $|N(y)\cap (\overline{S}\setminus\{v\})|\leq q$. Let's define $S' = ({S}\cup\{v\}) \setminus\{y\}$. Since $|S'|=|S|$ and $H[S']$ is $R$-free, we can conclude that $S'\in \F$. As $v$ belongs to a copy of $G$ in $H[\overline{S}]$ and $v\in B'$, we have $|N(v)\cap \overline{S}|=q$. According to $(P1)$, $y$ must be present in at least one copy of $G$ in $H[\overline{S'}]$. By applying $(P2)$, it can be easily verified that $|N(y)\cap (\overline{S}\setminus\{v\})|= q$.		
		\end{proof}
		Assume that $S=S_0$ and  $G_0$ is a copy of $G$ in $H[\overline{S}_0]$.  Assume that $V_0=V(G_0)=\{v_0^0,v_1^0,\ldots, v_{q}^0\}$, where $G_0[V_0\setminus \{v^0_{q-1}, v^0_{q}\}]\cong K_{q-1}$ and $G_0\cong K_{q-1}\oplus H[\{v^0_{q-1}, v^0_{q}\}]$. Now define  $B_0$  as follows
		\[ B_0\isdef\{v_0^0,v_1^0,\ldots, v_{q-2}^0\}=V_0\setminus\{v^0_{q-1}, v^0_{q}\}.\]
		
		Since $q\geq 4$, we can conclude that $|B_0|\geq 2$. Let's consider a vertex $v_0$ in $B_0$. Based on Claim \ref{c1}, $H[S_0\cup\{v_0\}]$ contains a unique copy of $R$, denoted as $R_0$, which is $(p-1)$-regular and one of the connected components of $H[S_0\cup\{v_0\}]$.
		
		Let's choose a vertex $y_0$ from $V(R_0)$ such that $y_0$ is not a cut vertex in $R_0$. Please note that we will use the assumption that $y_0$ is not a cut vertex in the remaining part of the proof.
		
		Define $S_1=(S_0\cup \{v_0\}) \setminus\{y_0\}$. It can be verified that $H[S_1]$ is $R$-free, and $|S_1|=|S_0|$. Hence, $S_1\in \F$. According to $(P1)$, and considering the fact that $v_0$ lies in at least one copy of $G$ in $\overline{S}_0$, we can deduce that $y_0$ must lie in a copy of $G$ in $H[\overline{S}_1]$, denoted as $G_1$.
		
		As $v_0\in B_0$, we have $|N(v_0)\cap \overline{S}_0|=q$. Therefore, utilizing $(P2)$, we can conclude that $|N(y_0)\cap (\overline{S}_0\setminus\{v_0\}) |= q$.
		
		Let's assume that $V_1=V(G_1)=\{v_0^1,v_1^1,\ldots, v_{q}^1\}$, and $G_1\cong K_{q-1}\oplus H[\{v^1_{q-1}, v^1_{q}\}]$, where $G_1[V_1\setminus \{v^1_{q-1}, v^1_{q}\}]\cong K_{q-1}$. Now, let's define $B_1$ as follows: if $v^1_{q-1}v^1_{q}\not \in E(H)$, then we define
		\[ B_1\isdef\{v_0^1,v_1^1,\ldots, v_{q-2}^1\}=V_1\setminus\{v^1_{q-1}, v^1_{q}\}\]  
		If $v^1_{q-1}v^1_{q}\in E(H)$, we proceed as follows: Let's assume that $W$ is a subset of $q-1$ elements from $V(G_1)$ such that $|W\cap B_0|$ has the maximum cardinality among all subsets of $q-1$ elements from $V(G_1)$. In this case, we define $B_1=W$.
		
		Now, let's present the following two claims based on the above construction:
		\begin{claim}\label{c3}  If $|(B_0\setminus\{v_0\})\cap (B_1\setminus\{y_0\})|\neq 0$, then  $B_0\setminus\{v_0\}= B_1\setminus\{y_0\}$.
		\end{claim}
		\begin{proof}[\bf Proof of Claim~\ref{c3}]Let's consider the case where $v_{q-1}^0v_{q}^0\notin E(H)$. In this situation, it can be observed that $\{v_{q-1}^0,v_{q}^0\}\nsubseteq B_1$ and $v_{q-1}^0,v_{q}^0\notin B_0$. Suppose there exists a vertex $z\in (B_0\setminus\{v_0\})\cap (B_1\setminus\{y_0\})$. Since $z\in B_0\setminus\{v_0\}$, we can deduce that $|N(z)\cap (\overline{S}_0\setminus \{v_0\})|= q-1$. Considering the fact that $z\in B_1$ and based on Part (I), we have $|N(z)\cap\overline{S}_1|= q$. As $y_0$ is adjacent to $z$ and $y_0\notin\{v_{q-1}^0,v_{q}^0\}$, it must be the case that $y_0\in B_1$. Consequently, we have $N(z)\cap(\overline{S_1}\setminus \{y_0\})= N(y_0)\cap(\overline{S_1}\setminus\{z\})$.
			
			Now, let's assume by contradiction that there exists a vertex $z'$ in $B_0\setminus \{v_0, z\}$ such that $z'\notin B_1\setminus \{y_0\}$. As $z'$ in $B_0\setminus \{v_0, z\}$, $v_{q-1}^0,v_{q}^0\in V(G_0)\setminus B_0$, we have $z'v_{q-1}^0,'v_{q}^0\in E(H)$. This implies that $|N(z)\cap \overline{S_1}|\geq q+1$, which is not possible. Hence, we can conclude that $B_0\setminus \{v_0\}=B_1\setminus \{y_0\}$.
			
			Now, let's consider the case where $v_{q-1}^0v_{q}^0\in E(H)$, which implies $H[V(G_0)]\cong K_{q+1}$. According to (II), $H[V(G_0)]$ is a connected component of $H[\overline{S}]$. Additionally, based on (P1), we can assume that $H[V(G_1)]\cong K_{q+1}$. Since $|(B_0\setminus\{v_0\})\cap (B_1\setminus\{y_0\})|\neq 0$, it can be verified that $V(G_0)\setminus\{v_0\}=V(G_1)\setminus\{y_0\}$. Therefore, considering this fact and the definition of $B_1$, we have $(B_1\setminus\{y_0\})=(B_0\setminus\{v_0\})$.
		\end{proof}
		
		\begin{claim}\label{c4} 
			If $|(B_0\setminus\{v_0\})\cap (B_1\setminus\{y_0\})|\neq 0$, then $K_{d+1}\setminus e\subseteq H$.
		\end{claim}
		\begin{proof}[\bf Proof of Claim~\ref{c4}]  Let's assume that $|(B_0\setminus\{v_0\})\cap (B_1\setminus\{y_0\})|\neq 0$. According to Claim~\ref{c3}, we have $B_0\setminus\{v_0\}= B_1\setminus\{y_0\}$. Now, let $z$ be a vertex in $B_1\setminus\{y_0\}$. It follows that $|N(z)\cap \overline{S}_1|= q$ and, consequently, $|N(z)\cap S_1|= p-1$. Based on Claim~\ref{c1}, we know that $S_1 \cup \{z\}$ contains a unique copy of $R$, denoted as $R_z$, such that $z$ lies in $R_z$, and $R_z$ is one of the connected components of $H[S_1\cup\{z\}]$. Considering the fact that $z$ is adjacent to $v_0$ in $H[\overline{S}_0]$, $y_0$ is not a cut vertex in $H[S_0\cup\{v_0\}]$, and $R_0$ is $(p-1)$-regular and one of the connected components of $H[S_0\cup\{v_0\}]$, we can conclude that $N(y_0)\cap (S_0\cup \{v_0\})= N(y_0)\cap V(R_0)= N(z)\cap S_1$. Therefore, we can deduce that $v_0$ is adjacent to $y_0$.
			
			We need to show that $H[N(y_0)\cap V(R_0)]\cong K_{p-1}$. Let's consider an arbitrary vertex $y \in N(y_0)\cap V(R_0)$. From the previous reasoning, we know that $N(z)\cap S_0= (N(y_0)\cap S_0)\cup \{y_0\}$ for each $z\in B_1\setminus\{y_0\}$.
			
			Now, let's define $S''\isdef S_0\cup\{v_0, z\}\setminus\{y\}$. Since $S_0$ is maximal, $H[S'']$ must contain at least one copy of $R$, denoted as $R'$. It's easy to see that $R'$ must contain $z$. As $R_0$ is $(p-1)$-regular and one of the connected components of $H[S_0\cup\{v_0\}]$, and $z$ has exactly $p-1$ neighbours in $S''$, we can conclude that $N(z)\cap V(R')= N(y)\cap V(R_0)$.
			
			Considering that $N(z)\cap S_0= (N(y_0)\cap S_0)\cup \{y_0\}$, we have $N(y)\cap (V(R_0)\setminus \{y_0\})=N(y_0)\cap (V(R_0)\setminus \{y\})$. Since $y$ is an arbitrary vertex in $N(y_0)\cap V(R_0)$, we can conclude that $H[N(y_0)\cap V(R_0)]\cong K_{p-1}$. Therefore, $R_0$ is isomorphic to $K_p$.

			To summarize the argument:
			
			We want to show that $K_{d-1}=K_{p+q-2}\subseteq H$. We already know that $R_0\cong K_p$ and that $N(y_0)\cap R_0= N(z)\cap S_1$ for each $z\in (B_0\setminus \{v_0\})\cap (B_1\setminus \{y_0\})$. Therefore, it is sufficient to prove that each vertex $y\in V(R_0)$ is adjacent to both $v_{q-1}^0$ and $v_q^0$.
			
			Assume by contradiction that there exists a vertex $y\in V(R_0)$ that is not adjacent to either $v_{q-1}^0$ or $v_q^0$. In this case, we define $S'\isdef (S_0\cup\{v_0\})\setminus\{y\}$. Since $q\geq 4$, if $v_{q-1}^0$ and $v_q^0$ are not adjacent, then $H[{\overline {S'}}]$ does not contain any copy of $G$, which contradicts (P1). Otherwise assume that $G'$ is a copy of $G$ in $H[\overline{S'}]$, which contains $y$. Hence one can say that  each vertex  $x\in B'\cap B_0$, has at least $q+1$ neighbours in $\overline{S'}$, which is not possible.
			
			Therefore, we conclude that $v_{q-1}^0$ and $v_q^0$ must be adjacent.
			
			Now, as $v_{q-1}^0$ and $v_q^0$ are adjacent, we have $H[V(G_0)]\cong K_{q+1}$. Since $y$ is not adjacent to at least one of $v_{q-1}^0$ and $v_q^0$, it follows that $H[\overline{S'}]$ does not contain any copy of $K_{q+1}$, which contradicts (P1).
			
			Therefore, in either case, we have shown that $K_{d+1}\subseteq H[N[v_0]]\subseteq H$. Combining this with the previous result that $K_{d-1}\subseteq H$, we conclude that $K_{d+1}\setminus \{e\}\subseteq H$.
		\end{proof}

		We can assume that $(B_0\setminus{v_0})\cap (B_1\setminus{y_0})= \varnothing$ based on Claim \ref{c4}.
		
		Now, let's consider the case when $i\geq 2$. We will assume that we have already defined $S_{i-1}$, $G_{i-1}$, $V_{i-1}$, $B_{i-1}$, $v_{i-1}$, ,$R_{i-1}$, and $y_{i-1}$, where:
		\begin{itemize}
			\item $S_{i-1}\in \F$,
			\item $G_{i-1}$ is a copy of $G$ in $\overline{S}_{i-1}$,
			\item $V_{i-1}=V(G_{i-1})=\{v^{i-1}_0,v^{i-1}_1,\ldots, v^{i-1}_{q}\}$ and $H[V_{i-1}\setminus \{v^{i-1}_{q-1}, v^{i-1}_{q}\}]\cong K_{q-1}=K'$.
			\item  $G_{i-1}\cong K'\oplus H[\{v^{i-1}_{q-1}, v^{i-1}_{q}\}]$,
			\item If $v^{i-1}_{q-1}v^{i-1}_{q}\notin E(H)$, then: $B_{i-1}=\{v^{i-1}_0,v^{i-1}_1,\ldots, v^{i-1}_{q-2}\}$,
			\item 	If $v^{i-1}_{q-1}v^{i-1}_{q}\in E(H)$, we proceed as follows: Let's assume that $W$ is a subset of $q-1$ elements from $V(G_{i-1})$ such that $|W\cap B_{j}|$ has the maximum cardinality among all subsets of $q-1$ elements from $V(G_{i-1})$, for each $j\leq i-2$. In this case, we define $B_{i-1}=W$.
			
			\item  $v_{i-1}$ in $B_{i-1}$,
			\item  $R_{i-1}$ is a copy of $R$ in $H[S_{i-1}\cup\{v_{i-1}\}]$.
		\end{itemize}
		As $S_{i-1}\in \F$, we can conclude that $H[S_{i-1}\cup\{v_{i-1}\}]$ contains a unique copy of $R$, denoted as $R_{i-1}$. This copy $R_{i-1}$ is one of the connected components in $H[S_{i-1}\cup\{v_{i-1}\}]$, as stated in Claim \ref{c3}. Let's assume that $y_{i-1}$ is a vertex in $R_{i-1}\setminus \{v_{i-1}\}$ and it is not a cut vertex.
		
		Next, we will define $S_i$, $B_i$, $G_i$, $V_i$, $v_i$, $R_i$, and $y_i$ for the next iteration $i$.
		
		We define $S_i$ as $S_i = (S_{i-1}\cup\{v_{i-1}\}) \setminus \{y_{i-1}\}$. Since $y_{i-1}\in R_{i-1}$ and $R_{i-1}$ is one of the connected components in $H[S_{i-1}\cup\{v_{i-1}\}]$, it follows that $S_i$ belongs to $\F$.
		
		Based on Claim \ref{c2}, let's assume that $G_i$ is a copy of $G$ in $H[\overline{S}i]$ that contains $y_{i-1}$. We can denote the vertex set of $G_i$ as $V_i = V(G_i) = \{v_0^i, v_1^i, \ldots, v_q^i\}$. Furthermore, we have $G_i\cong K_{q-1}\oplus H[{v_{q-1}^i, v_q^i}]$.
		Now, we define $B_i$ as follows:
		
		If ${v_{q-1}^i, v_q^i}\not \in E(H)$, then $B_i$ is given by $B_i = \{v_0^i, v_1^i, \ldots, v_{q-2}^i\} = V_i\setminus\{v_{q-1}^i, v_q^i\}$.\\
		If $v^i_{q-1} v^i_{q} \in E(H)$, we proceed as follows: 
		
		We assume that $W$ is a $(q-1)$-element subset of $V(G_i)$ such that $W$ has the maximum possible intersection with one of the sets $B_j$ for $0 \leq j \leq i-1$. We define $B_i$ as $B_i = W$.

		
		
		Since $H$ is a finite graph, there exists a minimum number $\ell \geq 2$ such that $(B_{\ell}\setminus\{y_{\ell-1}\})\cap (B_j\setminus\{v_j\}) \neq \varnothing$ for some $j \leq \ell-1$. Without loss of generality, we may assume that $j=0$.
		\begin{claim}\label{c6}   
			We have $B_{\ell}\setminus\{y_{\ell-1}\}\subseteq V(G_0)\setminus\{v_0\}$. In particular,
			If $v_{q-1}^0v_{q}^0\notin E(H)$, then $B_{\ell}\setminus\{y_{\ell-1}\}=   B_{0}\setminus\{v_{0}\}$.
		\end{claim}
		\begin{proof}[\bf Proof of Claim~\ref{c6}]  
			By the minimality of $\ell$, we have $B_{0}\setminus\{v_0\} \subseteq \overline{S}{\ell}$. We also aim to show that $\{v_{q-1}^0,v_{q}^0\} \subset \overline{S}_{\ell}$.
			
			Assume, to the contrary, that $v_{q}^0 \not\in \overline{S}{\ell}$. Therefore, there exists some $1 < i_0 < \ell$ such that $v_{q}^0 = v_{i_0} \in B_{i_0}$. Since $q \geq 4$, we have $B_{i_0} \cap B_0 \neq \varnothing$, which contradicts the minimality of $\ell$.

			Suppose $z\in (B_{0}\setminus\{v_0\})\cap B_{\ell}$. Vertex $z$ has $q$ neighbour in $G_{\ell}$. Now, assume to the contrary that there exists $w\in (B_{0}\setminus\{v_0\})\setminus B_{\ell}$.
			
			First, consider the case where $v_{q-1}^0v_{q}^0\notin E(H)$. Since the induced subgraph of $H$ on $B_{\ell}$ is isomorphic to $K_{q-1}$, we have $\{v_{q-1}^0,v_{q}^0\}\nsubseteq B_{\ell}$. We aim to show that $w$ does not belong to $N(v_{\ell})\setminus B_{\ell}$. Assume, by contradiction, that $w$ belongs to $N(v_{\ell})\setminus B_{\ell}$. Therefore, at least one of the vertices in $\{v_{q-1}^0,v_{q}^0\}=N(v_0)\setminus B_{0}$ does not belong to $N(v_{\ell})\cap \overline{S}{\ell}$. Consequently, $|N(z)\cap \overline{S}_{\ell}|\geq q+1$, which is not possible. Hence, $w$ does not belong to $N(v_{\ell})\setminus B_{\ell}$. As a result, $|N(z)\cap \overline{S}_{\ell}|\geq q+1$, which is not possible.
			
			Now, let's consider the case where $v_{q-1}^0v_{q}^0\in E(H)$. In this case, $H[V(G_0)]\cong K_{q+1}$ by condition (II). We can assume that $H[V(G_0)]\setminus\{v_0\}$ is a connected component of $H[\overline{S_{\ell}}]$. Since $z\in B_{\ell}\cap B_0$ and $y_{\ell-1}\in V(G_{\ell})$, $y_{\ell-1}$ must be adjacent to all neighbors of $z$ in $\overline{S_{\ell}}$, except possibly one of them.
			Therefore, we have $V(G_0)\setminus\{v_0\}=V(G_{\ell})\setminus\{y_{\ell-1}\}$. If $y_{\ell-1}$ is not adjacent to one of the neighbour of $z$ in $\overline{S_{\ell}}$, denoted as $v^q_{\ell}$, then we can set $B_{\ell} = V(G_0)\setminus\{v_0,v^q_{\ell}\}$. If $y_{\ell-1}$ is adjacent to all neighbours of $z$, then $G_{\ell}$ is isomorphic to $K_{q+1}$. In this case, we define $B_{\ell}=(B_0\setminus\{v_0\})\cup \{y_{\ell-1}\}$, and the proof is complete.	
		\end{proof}
		
		Now we are in a position to complete the proof of Theorem~\ref{M.th}.
		
		Recall that we have defined a sequence of sets $S_0, S_1, \ldots, S_{\ell}$ and corresponding subgraphs $G_0, G_1, \ldots, G_{\ell}$ in the graph $H$.
		
		We have shown that each set $S_i$ belongs to $\mathcal{F}$ and that $G_i$ is a copy of $G$ in $H[\overline{S_i}]$. Furthermore, we have defined sets $B_i$ for $0 \leq i \leq \ell$.
		
		By the construction of $B_i$, it follows that for $0 \leq i \leq \ell$, the subgraph induced by $B_i$ is isomorphic to $K_{q-1}$, except for the last step where it may be isomorphic to $K_{q}$ if $y_{\ell-1}$ is adjacent to all neighbours of $z$ in $\overline{S_{\ell}}$.
		
		We have also shown that $B_{\ell}\setminus\{y_{\ell-1}\} \subseteq \overline{S_{\ell}}$. Additionally, if $v_{q-1}^0 v_q^0 \in E(H)$, then $H[V(G_0)]\setminus\{v_0\}$ is a connected component of $H[\overline{S_{\ell}}]$, and we have $V(G_0)\setminus\{v_0\}=V(G_{\ell})\setminus\{y_{\ell-1}\}$.
		
		Therefore, we have successfully constructed a sequence of sets $S_0, S_1, \ldots, S_{\ell}$ and subgraphs $G_0, G_1, \ldots, G_{\ell}$ that satisfy the properties stated in Theorem~\ref{M.th}. Therefore, the theorem is proven.

			%
			%
		\begin{claim}\label{cc5} 	 The statement of  part ~(III) is true.
		\end{claim}
		\begin{proof}[\bf Proof of Claim~\ref{cc5}]
			
			By the maximality of $S_{\ell}$, for each $w \in B_{\ell} \cap (B_{0}\setminus\{v_0\})$, $S_{\ell} \cup \{w\}$ contains a unique copy of $R$ that contains $w$, denoted as $R_w$.
			This fact, combined with Claim \ref{c2}, implies that $v_0$ has exactly $p-1$ neighbour in $S_{\ell}\cup\{w\}$.
			
			On the other hand, since $R_0$ is $(p-1)$-regular and contains $v_0$, $v_0$ has exactly $p-1$ neighbour in $S_0$. Since $w$ is adjacent to $v_0$ in $H[S_{\ell}\cup\{w\}]$, there must exist $0 \leq i_0 < \ell$ such that $y_{i_0}$ is adjacent to $v_0$ in $R_{i_0}$. If there is no such $i_0$, then $|N(v_0)\cap (S_{\ell}\cup\{w\})|\geq p$, which contradicts the fact that $R_w$ is $(p-1)$-regular component of $H[S_{\ell}\cup\{w\}]$. Note that the minimality of $\ell$ and $\{v_0^{q-1},v_0^q\} \subset {\overline S_{\ell}}$ imply that there is only one such $i_0$.
			
			For each $w \in B_{0}\setminus\{v_0\}$, since $w$ is adjacent to $v_0$ in $H[S_{\ell}\cup\{w\}]$ and $R_{i_0}\setminus\{y_{i_0}\}$ is one of the connected components of $H[S_{\ell}\cup\{w\}]$, we have $N(w)\cap S_{\ell} = N(w)\cap (S_{i_0}\cup \{v_{i_0}\}) = N(y_{i_0})\cap (S_{i_0}\cup \{v_{i_0}\})$.
			
			Now, consider an arbitrary vertex $y \in N(w)\cap S_{\ell} = N(y_{i_0})\cap S_{i_0}$, and let $w'$ be a vertex in $B_0\setminus\{v_0,w\}$. Since $H[(S_{\ell}\cup\{w\})\setminus\{y\}]$ does not contain any copy of $R$, and $H[(S_{\ell}\cup\{w',w\})\setminus\{y\}]$ contains a copy of $R$, it can be shown that $N(w')\cap ((S_{\ell}\cup\{w\})\setminus\{y\}) = N(y)\cap (S_{\ell}\cup \{w\})$. Therefore, the subgraph induced by $N(w)\cap S_{\ell}$ is isomorphic to $K_{p-1}$.
			
			Since $H[S_{\ell}\cup\{w\}]$ contains a copy of $R$ and $w \in B_{\ell}$, by Claim \ref{c2}, every vertex $y \in N(w)\cap S_{\ell}$ must lie in at least one copy of $G$ in $H[({\overline S_{\ell}}\setminus\{w\})\cup{y}]$, and consequently, $y$ must be adjacent to all vertices of $G_{\ell}\setminus\{w\}$, except possibly $y_{\ell-1}$.
			
			Now, consider $w$ and its neighbour. It can be verified that $K_{p+q-1} \subseteq H[N(w)\cup\{w\}]$, which completes the proof.
			
			Note that if $v_{q-1}^0 v_{q}^0 \notin E(H)$, then $y$ must be adjacent to $y_{\ell-1}$, and as a result, a copy of $K_{p+q}\setminus e$ is contained in $H[N(w)\cup\{w\}]$. 
		\end{proof}	
\end{proof}
\section{Proof the Main result}
\begin{proof}[\bf Proof  Theorem \ref{mth1}]
	Suppose that $H = (V, E)$ is a connected graph with maximum degree $\Delta(H)\geq 5$  and  $H$ is $K_{\Delta(H)+1}\setminus e$-free. Suppose that $p$ and $q$  are two positive integers, such that $p\geq 2, q\geq 4$ and  $\Delta(H)+1=p+q$. Set $\G$ as a collection of   graphs with minimum degree at least $p-1$.  Consider a $(V_1,V_2)$-partition of $V(H)$, such that $H[V_1]$  is a maximum order $\G$-free induced subgraph of $H$. By maximality $V_1$ one can say that $\Delta(H[V_2])\leq q$. If $H[V_2]$  is  $K_q$-free or  its $q$-cliques are disjoint, then the proof is complete. Therefore, by contradiction suppose that there are at least two copies $K, K'$  of $K_q$ in $H[V_2]$,  such that $V(K)\cap V(K')\neq \emptyset$. Hence,  $\Delta(H[V_2])\leq q$,  implies that there exists at least one copy of $G$ in $H[V(K)\cup V(K')]\subseteq H[V_2]$, otherwise one can say that there is at least one member $v$  of $V(K)\cap V(K')$ such that $|N(v)\cap V_2|\geq q+1$, which is not possible. Hence by Theorem \ref{M.th}, we have $K_{\Delta(H)+1}\setminus e\subseteq H$, which would be a contradiction to the assumption. So, the theorem holds. 		 
\end{proof}
\begin{proof}[\bf Proof  Corollary \ref{cor1}] The proof is by induction on $k$. For $k=2$ the proof is complete by Theorem \ref{mth1}.	To prove the statement for $k\geq 3$, we
	set $p=p_1$ and $q=\sum_{i=2}^{k}p_i-(k-2)$. Note that $p+q=\Delta(H)+1$ and also   $p\geq 2, q\geq 4$.  Since the statement is true for $k=2$, we can obtain a partition of $V(H)$ into  $V_1$ and $V_2$ such that $H[V_1]$ is maximum $\G_1$-free, the maximum degree of $H[V_2]$ is at most $\Delta(H)-(p-1)=q$ and  $H[ V_2] $ is $K_q$-free  or  its $q$-cliques are disjoint. 
	
	If the maximum degree of $H[V_2]$ is less than $q$, let $v\in V_2$ be a vertex with maximum degree in $H[V_2]$ such that its degree is equal to $q'<q$. We add $q-q'$ new vertices to $H[V_2]$ and join all of them to $v$, forming a new graph $H'$. The graph $H'$ has maximum degree $q$ and is $K_q$-free  or  its $q$-cliques are disjoint. Since $k\geq 3$, $p_2\geq 2$ and $p_k\geq 4$, we have $\Delta(H')\geq 5$. Also as  the graph $H'$ is $K_q$-free  or  its $q$-cliques are disjoint one can say that $K_{\Delta(H')+1\setminus e}\nsubseteq H'$.  	 
	Therefore by induction there exists a partition of $V(H')$   into  $W_2,\cdots, W_{k}$ such that for each $2\leq i\leq k-1$, $H'[W_i]$ is $\G_i$-free, and the maximum degree of $H[W_k]$ is at most $p_k$ and  $H[ W_k] $ is $K_{p_k}$-free  or  its $p_k$-cliques are disjoint. 
	so $V_1, W_2\cap V_{2},\ldots, W_{k}\cap V_2$ is the desired partition of $V(H)$.
	
	Therefore, we may assume that $H[V_{2}]$  is a graph with maximum degree $q\geq 5$. We also have $\omega(H[V_{2}])\leq q$,   the maximum degree of $H[V_2]$ is at most $\Delta(H)-(p-1)=q$ and  $H[V_2] $ is $K_q$-free  or  its $q$-cliques are disjoint. We have $p_i\geq 2, ~p_k\geq 4$ and $\sum_{i=2}^{k}p_i=\Delta(H[V_2])-1+(k-1)$. Also as  the graph $H[V_2]$ is $K_q$-free  or  its $q$-cliques are disjoint one can say that $K_{\Delta(H[V_2])+1\setminus e}\nsubseteq H[V_2]$.  	 
	
	We may assume that $H[V_2]$ is a connected graph, otherwise if $H[V_2]$ has $\ell\ge 2$ connected components,
	say $H_1,\ldots, H_{\ell}$,  we prove the statement for each of these connected components. Then, for each connected component  $H_i$, 
	there exists a partition of $V(H_i)$ into  $V_{i,2}, \ldots V_{i,k}$  such that for each $2\leq t\leq k-1$, $H[V_{i,t}]$ is $\G_t$-free   and the maximum degree of $H[V_{i,k}]$ is at most $p_k$ and  $H[V_{i,k}] $ is $K_{p_k}$-free  or  its $p_k$-cliques are disjoint.
	Now, for each $2\leq j\leq k$ define $V_j=\cup_{i=1}^{\ell}V_{i,j}$. Therefore, $V_1,  V_{2},\ldots, V_k$ is the desired partition of $V(H)$.
	Hence the proof is complete. 
\end{proof}

To prove Lemma 	\ref{ml1} for the case tht $q=2,3$, we need the following theorem.
\begin{theorem}\label{th5}\cite{rabern2011hitting}
	If $H$ is a  graph with $\omega(H)\geq \frac{3(\Delta+1)}{4}$, then $H$ has an independent set $I$ such that $\omega(H\setminus I)\leq \omega(H)-1 $.
\end{theorem}
\begin{theorem}\label{th0}\cite{christofides2013note} Let $H$ is a connected graph with $\Delta(H)=\Delta$ and  $\omega(H)\geq\frac{2}{3}(\Delta +1)$, then $H$ contains an independent set intersecting every maximum clique unless it is the strong product of an odd hole and a clique.
\end{theorem}
In the next results by using Theorem \ref{th5} and Theorem \ref{th0} we prove that the Lemma \ref{ml1} is true for the case that $(p,q)=(d-1,2)$ and $(p,q)=(d-2,3)$.
\begin{proposition}\label{p1}
	Let $H$  is a connected graph with $\Delta(H)=d\geq 7$, $K_d\setminus e$-free and $\omega(H)=d-1$, also let  $(p,q)=(d-1,2)$ or $(p,q)=(d-2,3)$. Then there exists a $(V_1,V_2)$-partition of $V(H)$, so that $H[V_1]$ is $K_p$-free,  and $H[V_2] $ is $K_2$-free(~or $K_3$-free).	
\end{proposition}
\begin{proof} Assume that $(p,q)=(d-1,2)$. Since $d\geq 7$, according to Theorem \ref{th5}, we know that $H$ contains an independent set that intersects every maximum clique. Let's denote this stable set as $S$. By considering the graph $(H\setminus S, S)$, we complete the proof.
	
	Now, let's consider the case when $(p,q)=(d-2,3)$. Again, considering $d\geq 7$ and $\omega=d-1$, if $d\geq 8$, we can apply Theorem \ref{th5} to conclude that $H$ contains an independent set that intersects every maximum clique. Let's denote this stable set as $S_1$, and define $H_1=H\setminus S_1$. If $\omega (H_1)\leq p-3$, then the proof is complete. Otherwise, if $\omega(H_1)=d-2$, the maximum degree of $H_1$ is at most $d-1$. Let's assume it is exactly $d-1$. If not, we can introduce additional vertices and edges to obtain a graph $H'1$ that has maximum degree $d-1$ and is $K{d-1}$-free. Now, applying Theorem \ref{th5} to $H'_1$, we can find another stable set $S_2$ that intersects every maximum clique in $H_1$. Define $H_2=H_1\setminus S_2$. Since $\omega (H_2)\leq p-3$, the proof is complete by considering the graph $(H\setminus (S_1\cup S_2), S_1\cup S_2)$.
	
	Finally, let's consider the case when $d=7$. According to Theorem \ref{th5}, we know that $H$ contains an independent set that intersects every maximum clique.Assume that $S_1$ is the stable set obtained from Theorem \ref{th5}. Let $H_1=H\setminus S_1$. We can assume that $\omega(H_1)=5$, as otherwise, the proof is complete. Since $\omega(H_1)=5$, we can verify that $H_1$ is not the strong product of an odd hole and a clique. Consequently, by Theorem \ref{th0}, $H_1$ contains a stable set that intersects every maximum clique. Let $S_2$ be this stable set. Therefore, the proof is complete by considering the graph $(H\setminus (S_1\cup S_2), S_1\cup S_2)$. Thus, the proposition holds.
\end{proof}

\begin{proof}[\bf Proof   Lemma	\ref{ml1}]
	As $H$ is $ K_{d}\setminus e$-free,  by Theorem \ref{M.th},  it can be checked that  Lemma \ref{ml1}  holds for the case that $q\geq 4$, also by Proposition \ref{p1}, Lemma \ref{ml1}  holds for the case that $q=2,3$, hence the proof is complete.  
\end{proof}
According to Theorem \ref{mth1}, for any graph $H$ with a maximum degree of $5$ that is $K_5\setminus e$-free, there exists a partition $(V_1, V_2)$ of $V(H)$ such that $H[V_1]$ is a maximum acyclic induced subgraph, and $\omega(H[V_2])$ as well as $\Delta(H[V_2])$ are both at most $3$. Moreover, $H[V_2]$ is either $K_3$-free or its $3$-cliques are disjoint, which is an improvement over Theorem \ref{th1}.

Similarly, according to Theorem \ref{mth1}, for any graph $H$ with a maximum degree of $d \geq 5$ that is $K_{d+1}\setminus e$-free, there exists a partition $(V_1, V_2)$ of $V(H)$ such that $H[V_1]$ is a maximum acyclic induced subgraph, and $\omega(H[V_2])$ as well as $\Delta(H[V_2])$ are both at most $d-2$. Additionally, $H[V_2]$ is either $K_{d-2}$-free or its $(d-2)$-cliques are disjoint. In fact, if we take $\mathcal{G}=\mathcal{C}={C_n | n\geq 3}$, then Theorem \ref{mth1} encompasses and improves upon Theorem \ref{th1}. Hence, it is evident that this result is superior to the result of Theorem \ref{th1}.

\subsection{ Some research problems related to the contents of this paper.}

In this section, we propose some research problems related to the contents of this paper. The first problem concerns  Theorem \ref{mth1}, as we address below: 
\begin{problem}
	Suppose that $H = (V, E)$ is a connected graph with maximum degree $\Delta(H)\geq 5$  and  $H$ is $K_{\Delta(H)+1}\setminus e$-free. Suppose that $p$ and $q$  are two positive integers, such that $p\geq 2, q= 3$ and  $\Delta(H)+1=p+q$. Set $\G$ as a collection of   graphs with minimum degree at least $p-1$. Then there exists a $(V_1,V_2)$-partition of $V(H)$, such that $H[V_1]$  is a maximum order $\G$-free induced subgraph of $H$,   $\Delta(H[V_2])\leq q$, and  either $H[V_2] $ is $K_q$-free subgraph or its $q$-cliques  are disjoint.
	
\end{problem}

\begin{problem}
	Let $p,q$ be two positive integers,   where   $ p\geq 2, q\geq 3$, and $p+q=d+1$. Set $\G$ as a collection of some $(p-1)$-regular graph, $\G'$ as a collection of some $(q-1)$-regular graph, and suppose that $H$  is a connected graph with  $\Delta(H)=d\geq 5$ and $\omega(H)= \omega\leq d-2$. Then there exist $(V_1,V_2)$-partition of $V(H)$, so that $H[V_1]$ is $\G$-free, $V_1$ has the M-P-size,  and $H[V_2] $ is $\G'$-free.
	
\end{problem}

\begin{problem}
	Let $p,q$ be two positive integers,   where   $ p\geq 2, q\geq 4$, and $p+q=d+1$. Suppose that $H$ is a graph with  $\Delta(H)=d\geq 5$ and $\omega(H)\leq d-2$. If $d= p+q+1$, then there exists a $(V_1,V_2)$-partition of $V(H)$, such that $H[V_1]$ is a maximum $(p-2)$-degenerate induced subgraph and $H[V_2]$ is $(q-2)$-degenerate. 
	
\end{problem}

\section{Declarations}
{\bf Conflict of Interest:} On behalf of all authors, the corresponding author states	that there is no conflict of interest.

{\bf Data Availability Statement:}	 No data were generated or used in the preparation of this paper.
\bibliographystyle{plain}
\bibliography{v.p}

\begin{thebibliography}{10}

\bibitem{abu2020partitioning}
Faisal~N Abu-Khzam, Carl Feghali, and Pinar Heggernes.
\newblock Partitioning a graph into degenerate subgraphs.
\newblock {\em European Journal of Combinatorics}, 83:103015, 2020.

\bibitem{bickle2012structural}
Allan Bickle.
\newblock Structural results on maximal k-degenerate graphs.
\newblock {\em Discussiones Mathematicae Graph Theory}, 32(4):659--676, 2012.

\bibitem{bickle2021maximal}
Allan Bickle.
\newblock Maximal k-degenerate graphs with diameter 2.
\newblock {\em Mathematical Combinatorics}, 2:68--79, 2021.

\bibitem{bickle2021wiener}
Allan Bickle and Zhongyuan Che.
\newblock Wiener indices of maximal k-degenerate graphs.
\newblock {\em Graphs and Combinatorics}, 37(2):581--589, 2021.

\bibitem{Bollob}
B\'{e}la Bollob\'{a}s and Bennet Manvel.
\newblock Optimal vertex partitions.
\newblock {\em Bull. London Math. Soc.}, 11(2):113--116, 1979.

\bibitem{Catlin}
Paul~A. Catlin.
\newblock Brooks' graph-coloring theorem and the independence number.
\newblock {\em J. Combin. Theory Ser. B}, 27(1):42--48, 1979.

\bibitem{Catlin1}
Paul~A. Catlin and Hong-Jian Lai.
\newblock Vertex arboricity and maximum degree.
\newblock {\em Discrete Math.}, 141(1-3):37--46, 1995.

\bibitem{christofides2013note}
Demetres Christofides, Katherine Edwards, and Andrew~D King.
\newblock A note on hitting maximum and maximal cliques with a stable set.
\newblock {\em Journal of Graph Theory}, 73(3):354--360, 2013.

\bibitem{lick1970k}
Don~R Lick and Arthur~T White.
\newblock k-degenerate graphs.
\newblock {\em Canadian Journal of Mathematics}, 22(5):1082--1096, 1970.

\bibitem{MR2327961}
Mart\'{\i}n Matamala.
\newblock Vertex partitions and maximum degenerate subgraphs.
\newblock {\em J. Graph Theory}, 55(3):227--232, 2007.

\bibitem{rabern2011hitting}
Landon Rabern.
\newblock On hitting all maximum cliques with an independent set.
\newblock {\em Journal of Graph Theory}, 66(1):32--37, 2011.

\bibitem{rowshan2022catlin}
Yaser Rowshan and Ali Taherkhani.
\newblock A catlin-type theorem for graph partitioning avoiding prescribed
  subgraphs.
\newblock {\em Discrete Mathematics}, 345(8):112911, 2022.

\bibitem{yas}
Yaser Rowshan and Ali Taherkhani.
\newblock Partitioning of a graph into induced subgraphs not containing
  prescribed cliques.
\newblock {\em arXiv preprint arXiv:2210.04967}, 2022.

\bibitem{simoes1976joins}
JMS Simoes-Pereira.
\newblock Joins of n-degenerate graphs and uniquely (m, n)-partitionable
  graphs.
\newblock {\em Journal of Combinatorial Theory, Series B}, 21(1):21--29, 1976.

\bibitem{von2022point}
Justus von Postel, Thomas Schweser, and Michael Stiebitz.
\newblock Point partition numbers: decomposable and indecomposable critical
  graphs.
\newblock {\em Discrete Mathematics}, 345(8):112903, 2022.

\end{thebibliography}
\end{document}